\documentclass[11pt,reqno]{amsart}

\usepackage[T2A]{fontenc}
\usepackage[russian,english]{babel}
\usepackage[cp1251]{inputenc}
\usepackage{izvuz7}
\usepackage{amssymb,amsmath,amsthm,amsfonts}
\usepackage{graphicx}

\theoremstyle{plain}
\newtheorem{theorem}{Theorem}

\theoremstyle{definition}
\newtheorem{definition}{Definition}
\newtheorem{corollary}{Corollary}
\newtheorem{remark}{Remark}
\newtheorem{example}{Example}

\tit{\MakeUppercase{Consistent rational approximations of power series,\\ trigonometric series and series of Chebyshev polynomials}}
\shorttit{\MakeUppercase{Consistent rational approximations}}

\author{\MakeUppercase{A.P.\,Starovoitov, I.V.\,Kruglikov, T.M.\,Osnach}}
 
\god{202\_}
\nomer{\textnumero\,\_}

\pp{\pageref{firstpage}--\pageref{lastpage}}

\begin{document}
\selectlanguage{english}
\maketit
\label{firstpage}

\vskip0.3truecm
\hbox to \hsize{\hss {\vbox{
			\noindent \small {\sl \small Abstract. 
				For trigonometric series and series of Chebyshev polynomials, we defined trigonometric Hermite–Pad\'e and Hermite–Jacobi approximations, linear and nonlinear Hermite--Chebyshev approximations. We established criterion of the existence and uniqueness of trigonometric Hermite–Pad\'e polyno\-mials, associated with an arbitrary set of $k$ trigonometric series, and we found explicit form of these polynomials. Similar results were obtained for linear Hermite--Chebyshev approximations. We made examples of systems of functions for which trigonometrical Hermite–Jacobi approximations are existed but aren’t the same as trigonometric Hermite--Pad\'e approximations. Similar examples were made for linear and nonlinear Hermite--Chebyshev approximations.}
	}}\hss} 

\vskip0.2truecm
\hbox to \hsize{\hss {\vbox{
			\selectlanguage{english}
	}}\hss}

\section*{Introduction}
The theory of Pad\'e approximations is one of the intensively developing sections of complex analysis. The construction of such approximations is in a sense universal: their analogues can be constructed, in particular, for trigonometric series, series of Chebyshev polynomials and Faber polynomials. In recent years, generalizations of this kind have found numerous applications in various problems of analysis, applied mathematics, theoretical physics, mechanics, geophysics
(\cite{Andr1}--\cite{Gon1}). In this regard, for example, in \cite{Suet2} notes that the Pad\'e-Chebyshev approximations, along with the classical Pad\'e approximations have actually become an integral part of scientific and technical calculations, which is reflected in the creation of special programs for their location in well-known computer systems (\cite{Tee, Prox}).

We owe the appearance of two different approaches to determining Pad\'e approximations to G.\,Frobenius~\cite{Frob} and C.\,Jacobi~\cite{Jacobi}.  
These two approaches, which generally lead to different approximations, are retained when constructing their analogues for trigonometric series and series in Chebyshev polynomials. Moreover, if for power series the differences in questions of existence and uniqueness between two types of such approximations are well known (\cite{Bek}), then for the generalizations under consideration only the first steps have been made in this direction (for example,   \cite{Suet2, Bek, SuetCh}).

Similar design to C.\,Jacobi~\cite{Jacobi}, but designed for simultaneous interpolation 
several functions, was developed by Ch.\,Hermite~\cite{Hermite} and formed the basis of his proof of the transcendence of number $e$. In the work~\cite{Hermite} rational fractions appeared for the first time, which are now commonly called  {\it Hermite--Pad\'e approximations}. It should be said that Hermite--Pad\'e approximations for a system of functions can also be of two types. In what follows we will call them {\it Hermite--Pad\'e approximations} ({\it Hermite--Frobenius--Pad\'e approximations}) and {\it Hermite--Jacobi approximations}. Initially, Ch.\,Hermite considered the simultaneous interpolation of several exponential functions. In this case, both constructions of rational fractions coincide. In the case of an arbitrary system of functions this is not the case. We also note that at present, Hermite--Pad\'e approximations have been studied quite well (for example, ~\cite{Bek, Mahler,  NikSor}) and have found numerous applications in various areas of algebra, analysis and modern physics (for example, \cite{Beckermann}--\cite{Assche}). While the main properties of their generalizations, for example, for trigonometric series and series of Chebyshev polynomials, as far as we know, remain unexplored.

The present work is devoted to the construction of analogues of the Hermite--Pad\'e and Hermite--Jacobi approximations for trigonometric series and series of Chebyshev polynomials. From a formal point of view, the definition of new constructions of rational fractions for such series does not present any particular difficulties. At the same time, obtaining any meaningful results on the existence and uniqueness of new approximations is not as simple as with rational approximations of power series (\cite{Gon1, Suet2, Bek, Geddes, Litvinov, Adukov}).

\section{Hermite--Pad\'e and Hermite--Jacobi approximations} 
\subsection{Hermite--Pad\'e approximations}\label{s2.1}
Let us present some well-known facts of the theory of Hermite--Pad\'e approximations, which we will need later.

For an arbitrary function $f(z)$, represented by a power series
\begin{equation}
\label{eq2.1}
f(z)=\sum_{l=0}^{\infty}f_lz^{l},                           
\end{equation}
and for every pair $(n,m)$ of non-negative integers there exist algebraic polynomials $Q_m(z)$, $P_n(z)$,   $\deg Q_m \leqslant{m}$, $\deg P_n \leqslant{n}$,  
for which
\begin{equation}
\label{eq2.2}
Q_m(z)f{(z)}-P_n(z)=O(z^{n+m+1})\,\,.
\end{equation}
Here and below, by $O(z^p)$ we mean a power series of the form $c_1z^p+c_2z^{p+1}+\ldots$\,\,\,.

Polynomials $Q_m(z)$ and $P_n(z)$ are not uniquely determined by condition~\eqref{eq2.2}, but fractions
$$
\pi_{n,\,m}(z)=\pi_{n,\,m}(z;f)=\frac{P_n(z)}{Q_m(z)}
$$
define the same rational function no matter what polynomials $Q_m(z)$, $P_n(z)$,  satisfying~\eqref{eq2.2}, we take \cite[chapter~2, \S 1]{NikSor}. Rational functions $\pi_{n,\,m}(z)$ are usually called  {\it Pad\'e or Frobenius--Pad\'e approximations} (as a definition of a rational fraction $\pi_{n,\,m}(z;f)$ the relations~\eqref{eq2.2} were first proposed in 1881 by G.\,Frobenius \cite{Frob}, the authorship of Pad\'e is based on his dissertation \cite{PadD1892} of 1892).

Even earlier (in 1846), a problem similar in formulation was studied by С.\,Jacobi~\cite{Jacobi}, who generalized O.\,Cauchy's result on rational interpolation of a function defined at $n+m+1$ different points. С.\,Jacobi considered $(n+m+1)$-fold rational interpolation at one point. His interpolation construction leads to the following definition.

\begin{definition}\label{de1}
	A rational fraction
	$$
	\widehat{\pi}_{n,\,m}(z)=\widehat {\pi}_{n,\,m}(z;f)=\frac{\widehat{P}_n(z)}{\widehat{Q}_m(z)}\,,
	$$
	in which the algebraic polynomials $\widehat{Q}_{m}(z),\widehat{P}_{n}(z)$ have degrees not higher than $m$ and $n$, respectively, will be called  {\it Pad\'e--Jacobi approximation} for a pair $(n,m)$ and a function $f(z)$, if 
	$$
	f(z)-\frac{\widehat{P}_n(z)}{\widehat{Q}_m(z)}=O(z^{n+m+1}).
	$$
\end{definition}

Unlike Pad\'e approximations $\pi_{n,\,m}(z;f)$, which exist for any pair of indices $(n,m)$, Pad\'e--Jacobi approximations $\widehat {\pi}_{n,\,m}(z;f)$ may not exist. A corresponding example of a function $f(z)$ and a pair of indices $(n,m)$ is given in \cite[chapter 1, \S\,1.4]{Bek}. The first significant result in the study of the conditions under which $\widehat {\pi}_{n,\,m}(z;f)$ exists, was obtained by C.\,Jacobi~\cite{Jacobi}. To formulate it, we introduce Hadamard determinants 
\begin{equation}
\label{eq2.3}
H_{n,m}= \left| \!\!\! \begin{array}{cccc}
f_{n-m+1} & f_{n-m+2} & \ldots & f_{n} \\
f_{n-m+2} & f_{n-m+3} & \ldots & f_{n+1} \\
\ldots & \ldots & \ldots & \ldots \\
f_{n} & f_{n+1} & \ldots & f_{n+m-1} \\
\end{array} \!\!\! \right|\,,
\end{equation}
the elements of which are the coefficients of series~\eqref{eq2.1}. Here for $p<0$ we assume that $f_p=0$. 

\begin{theorem}[Jacobi~\cite{Jacobi}]\label{t1}
	If for a pair $(m,n)$ the determinant $H_{n,m}\neq0$, then the Pad\'e--Jacobi approximations $\widehat {\pi}_{n,\,m}(z;f)$ exist and coincide with the Frobenius--Pad\'e approximations of the function $f(z)$, i.e.  $$\widehat {\pi}_{n,\,m}(z;f)={\pi}_{n,\,m}(z;f).$$
\end{theorem}

A complete study of the conditions under which $\widehat {\pi}_{n,\,m}(z;f)$ exist was conducted by G.\,Baker~\cite[chapter 1, \S\,1.4]{Bek}. In this connection, rational functions $\widehat {\pi}_{n,\,m}(z;f)$ are also called Pad\'e approximations in the sense of Baker.

\vspace{0.1 cm}
Let us now consider system ${\bf f}=(f_1,\ldots,f_k)$, consisting of
$k$ functions, representable by power series
\begin{equation*}
\label{eq1.4}
f_j(z)=\sum_{l=0}^{\infty}{f^j_l}{z^{l}},\,\,\,\,\,
j=1,2,\ldots,k\,
\end{equation*}
with complex coefficients. 
The set of $k$--dimensional
multi-indices, which are an ordered set of $k$ non-negative integers, is denoted by $\mathbb{Z}^k_+$. The order of the multi-index $\overrightarrow{m}=(m_1,\ldots,m_k)\in \mathbb{Z}^k_+$
is the sum $m=m_1+\ldots+m_k$.

 Let us fix an index $n\in
\mathbb{Z}^1_+$ and a multi-index ${\overrightarrow{m}=(m_1,
	\ldots, m_k)\in\mathbb{Z}^k_+ }$ and consider the problem first posed and solved by Ch.\,Hermite for a set of exponents $e^{x},\,e^{2x},\ldots,e^{kx}$, and for an arbitrary system of analytic functions, formulated in its final form by K.\,Mahler~\cite{Mahler}. Let's use her formulation in the monograph~\cite[chapter 4, \S 1, problem А]{NikSor}:

\vspace{0,2cm} {\bf Problem А}. {\it Find a polynomial
	$Q_m(z)=Q_{n,\overrightarrow{m}}(z;{\bf f})$ that is not identically equal to zero and whose degree $\deg Q_m \leqslant m$, and polynomials $P^j_{n_j}(z)=P^j_{n,\overrightarrow m}(z;{\bf f})$, 
	$\deg P^j_{n_j}\leqslant n_j$, $n_j=n+m-m_j$, so that for}
$j=1,2,\ldots,k$
\begin{equation}
\label{eq2.5}
R^j_{n,\overrightarrow{m}}(z):=
Q_m(z)f_j(z)-P_{n_j}^j(z)=O(z^{n+m+1}).
\end{equation}

\begin{definition}\label{de2}
	Polynomials $Q_m(z), P^1_{n_1}(z),\ldots, P^k_{n_k}(z)$, that are a solution to problem~{\bf A} (a solution to problem~{\bf A} always exists~\cite{NikSor}), 
	and rational fractions
	$$
	\pi_j(z;{\bf f})=	\pi^j_{n_j,n,\overrightarrow{m}}(z;{\bf f})=\frac{P^j_{n_j}(z)}{Q_m(z)},\,\,\,j=1,2,\ldots,k\,
	$$
	are called, respectively, {\it Hermite--Pad\'e polynomials} and   {\it   Hermite--Pad\'e approximations} for a multi-index $(n,\overrightarrow{m})$ and a system   ${\bf f}$.
\end{definition}

Of particular interest are systems of functions $\bf f$, for which rational fractions $\{\pi_j(z;{\bf f})\}_{j=1}^k$ are uniquely determined by conditions~\eqref{eq2.5} for any multi-index $(n,\overrightarrow{m})$.  An important example of such systems are perfect systems~\cite[chapter 4, \S 1]{NikSor}. A system of exponents ${\bf E_k}=\left\{ e^{\lambda_jz}\right\}^k_{j=1}$ is perfect, where $\lambda_j$ are different non-zero complex numbers. Without a formal definition, this statement was proven by Ch.\,Hermite~\cite{Hermite} (\cite[chapter 4, \S 2]{NikSor}). Examples of systems other than perfect ones, for which $\{\pi_j(z;{\bf f})\}_{j=1}^k$  are determined uniquely, are given in \cite{StarRjab1, StarRjab2}. In these works, in particular, necessary and sufficient conditions were found under which problem A has a unique (up to a numerical factor) solution.

\subsection{Hermite--Jacobi approximations}\label{s2.2}
Let us introduce into consideration multiple analogues of C.\,Jacobi fractions~\cite{Jacobi}.

\begin{definition}\label{de4}
	Rational functions of the form
	$$
	\widehat{\pi}_j(z;{\bf f})=	\widehat{\pi}_{n_j,n,\overrightarrow{m}}(z;{\bf f})=\frac{\widehat{P}^j_{n_j}(z)}{\widehat{Q}_m(z)},\,\,\,j=1,2,\ldots,k\,,
	$$
	where algebraic polynomials $\widehat{Q}_m(z)=\widehat{Q}_{n,\overrightarrow{m}}(z;{\bf f})$,
	$\widehat{P}^j_{n_j}(z)=\widehat{P}_{n,\overrightarrow{m}}^j(z;{\bf f})$ have degrees not higher than $m$ and $n_j$ respectively, $n_j=n+m-m_j$, will be called {\it  Hermite--Jacobi approximations} for a multi-index $(n,\overrightarrow{m})$ and a system of functions ${\bf f}$, if 
	\begin{equation}
	\label{eq2.9}
	f_j(z)-\frac{\widehat{P}^j_{n_j}(z)}{\widehat{Q}_m(z)}=O(z^{n+m+1}).
	\end{equation}
\end{definition}

Polynomials $\widehat{Q}_m(z), \widehat{P}^1_{n_1}(z),\ldots,\widehat{P}^k_{n_k}(z)$, satisfying conditions~\eqref{eq2.9} (we will call them {\it  Hermite--Jacobi polynomials}), as well as Hermite--Jacobi approximations, may not exist. Let us give a corresponding example.

\begin{example}
	\label{ex1}
Let $k=2$, $n=1$, $m_1=m_2=1$,
$$
f_1(z)=2+z+2z^2+z^3+2z^4+z^5+\ldots,
$$
$$
f_2(z)=1+z+2z^2+3z^3+4z^4+5z^5+\ldots\,.
$$
Then $m=2$, $n_1=n_2=2$, and polynomials $\widehat{Q}_2(z), \widehat{P}^1_{2}(z),\widehat{P}^2_{2}(z)$, satisfying conditions~\eqref{eq2.9}, with a suitable choice of the normalizing factor, must necessarily have the form:
$$
\widehat{Q}_{2}(z)=z-2z^2,\,\, \widehat{P}^1_{2}(z)=2z-3z^2,\,\,\widehat{P}^2_{2}(z)=z-z^2.
$$
It is easy to check that, in fact, for the found polynomials, conditions~\eqref{eq2.9} are not satisfied, i.e.
$$
f_1(z)-\dfrac{2z-3z^2}{z-2z^2}\neq O(z^4),\,\,\,f_2(z)-\dfrac{z-z^2}{z-2z^2}\neq
O(z^4).
$$
\end{example}

\begin{remark}\label{re1}
	It follows from the definitions that if Hermite--Jacobi approximations exist for $(n,\overrightarrow{m})$ and ${\bf f}$, then they are also Hermite--Pad\'e approximations. Example 1 shows that the converse statement is, generally speaking, not true: polynomials $\widehat{Q}_2(z), \widehat{P}^1_{2}(z),\widehat{P}^2_{2}(z)$ are Hermite--Pad\'e polynomials, but are not Hermite--Jacobi polynomials for the systems of functions and the multi-index under consideration.
\end{remark}

\subsection{Generalization of Jacobi's theorem}\label{s2.3}
Let us introduce new notations. For index $n\in
\mathbb{Z}^1_+$ and non-zero multi-index ${\overrightarrow{m}=(m_1,
	\ldots, m_k) }$ consider a determinant

\begin{equation*}
H_{n,\overrightarrow m}=\det
\left(  \begin{array}{cccc}
H^1 \\
H^2 \\
\vdots  \\
H^k \\
\end{array}  \right)
=
\left|  \begin{array}{cccc}
f_{n-m_1+1}^{1} & f_{n-m_1+2}^{1} & \ldots & f_{n_1}^{1} \\
\ldots & \ldots & \ldots & \ldots \\
f_{n}^{1} & f_{n+1}^{1} & \ldots & f_{n+m-1}^{1} \\
\ldots & \ldots & \ldots & \ldots \\
f_{n-m_k+1}^k & f_{n-m_k+2}^k & \ldots & f_{n_k}^k \\
\ldots & \ldots & \ldots & \ldots \\
f_{n}^k & f_{n+1}^k & \ldots & f_{n+m-1}^k \\
\end{array}  \right|,
\end{equation*}
for $m_j\neq0$ consisting of blocks
$$
H^j=\left( \! \begin{array}{cccc}
f_{n-m_j+1}^{j} & f_{n-m_j+2}^{j} & \ldots & f_{n_j}^{j} \\
f_{n-m_j+2}^{j} & f_{n-m_j+3}^{j} & \ldots & f_{n_j+1}^{j} \\
\ldots & \ldots & \ldots & \ldots \\
f_{n}^j & f_{n+1}^j & \ldots & f_{n+m-1}^j \\
\end{array} \! \right),
$$
located one above the other. For $p<0$ here we also assume $f^j_p=0$. By definition, we assume that for $m_j=0$, determinant $H_{n,\overrightarrow m}$ does not contain block $H^j$. Determinant $H_{n,\overrightarrow m}$ is a multiple analogue of the Hadamard determinant~\eqref{eq2.3}: if $k=1$ or ${\overrightarrow{m}=(m_1,0,\ldots, 0) }$, then $H_{n,\overrightarrow m}$ exactly coincides with the Hadamard determinant~\eqref{eq2.3} of function $f_1(z)$.

The following multiple analogue of Jacobi's theorem holds~\cite{OsnStarRjab}.

\begin{theorem}[\cite{OsnStarRjab}]\label{t2}
	Let the determinant $H_{n,\overrightarrow m}\neq 0$ for the multi-index $(n,\overrightarrow m)\in\mathbb{Z}^{k+1}_+ $, $\overrightarrow{m}\neq (0,\ldots,0)$ and the system of functions ${\bf f}$. Then for $(n,\overrightarrow m)$ and ${\bf f}$ the Hermite--Jacobi approximations exist, are uniquely defined, and each of them identically coincides with the corresponding Hermite--Pad\'e approximation, i.e.
	\begin{equation}
	\label{eq2.10}
	\widehat\pi_{j}(z;{\bf f})=	\pi_{j}(z;{\bf f}),\,\,\,\,\, j=1,2,\ldots,k.
	\end{equation}	
\end{theorem}

	In the formulation of the theorem~\ref{t2} it was assumed that the multi-index $\overrightarrow{m}$ is non-zero. If $\overrightarrow{m}=(0,\ldots,0)$, then up to a numerical factor $\widehat{Q}_m(z)\equiv Q_m(x)\equiv 1$, and polynomials
	$\widehat{P}^j_{n}(z)$, $ P^j_{n}(z)$  coincide and are the
	$n$-th partial sum of series $f_j(z)$. Therefore, in this case, equalities~\eqref{eq2.10} are also preserved.
	
	The proof of theorem~\ref{t2}
is similar to the proof of Jacobi's theorem. (\cite[chapter 1, \S 1.1]{Bek}):  if $H_{n,\overrightarrow m}\neq 0$, then problem A has a unique solution (\cite{StarRjab2}), and under some normalization  $Q_m(z)=b_0+b_1z+\ldots+b_mz^m$, where $b_0=H_{n,\overrightarrow m}\neq 0$, therefore fraction $1/Q_m(z)$ is analytic in some neighborhood of zero.

\begin{definition}\label{de5}
	We call the system ${\bf f}$ {\it quite perfect}, if for any multi-index $(n,\overrightarrow m)\in\mathbb{Z}^k_+$ and $j=1,2,\ldots,k$
	$$
	\deg Q_m=m,\,\, \deg P_{n_j}^j=n_j, \,\,\gcd(Q_m,P^j_{n_j})=1.
	$$	
\end{definition}
 In~\cite{StarRjab2} it is proved that if the system is quite perfect, then for any multi-index $(n,\overrightarrow m)$ $H_{n,\overrightarrow m}\neq 0$.
Therefore, the following corollary is true.
 
\begin{corollary}
 If system ${\bf f}$ is quite perfect, then for any multi-index $(n,\overrightarrow m)$ there exist Hermite--Jacobi approximations and equalities~\eqref{eq2.10} are valid. 
\end{corollary}

A quite perfect system is, for example, the above-mentioned system of exponents ${\bf E_k}=\left\{ e^{\lambda_jz}\right\}^k_{j=1}$. 
Let us consider another example of a system (\cite{Aptek81M-L}), satisfying the conditions of the theorem~\ref{t2}, which we will need later. 

Let ${\bf E_{k,\gamma}}=\{E_{\gamma}(\lambda_j z)\}^k_{j=1}$, where $E_{\gamma}(z)$, being simultaneously a degenerate hypergeometric function and a Mittag--Leffler function, be represented by a power series
\begin{equation}
\label{eq5.1}
E_{\gamma}(z)=\,_1F_1(1,\gamma;z)=
\sum_{p=0}^{\infty}\frac{z^p}{(\gamma)_p}\,.                                   \end{equation}
In definition ${\bf E_{k,\gamma}}$ and \eqref{eq5.1} it is assumed that the parameter $\gamma\in\mathbb{R}\setminus
\mathbb{Z_{-}}$, $\mathbb{Z_{-}}=\{0,-1,-2,\ldots\}$, $(\gamma)_0=1$,
$(\gamma)_p=\gamma(\gamma+1)\cdot\cdot\cdot(\gamma+p-1)$ --- Pochhammer symbol, and $\{\lambda_j\}_{j=1}^k$ --- distinct non-zero real numbers (for $k=1$ we assume that $\lambda_1=1$). 

Functions $E_{\gamma}(z)$ were introduced into consideration by M.\,Mittag--Leffler~\cite{Lef} as a generalization of the exponential function, in particular, $E_{1}(z)=\exp z$.
Uniform convergence of Hermite--Pad\'e approximations $\pi_j(z;{\bf E_{k,\gamma}})$ to function $E_{\gamma}(\lambda_j z)$ on compact sets in $\mathbb{C}$ for $k=1$ was proved by M.\,De~Bruyne~\cite{Bruin}, and for $k>1$ by A.I.\,Aptekarev~\cite{Aptek81M-L}. The asymptotics of this convergence was investigated in \cite{StarStek}.  In~\cite{Aptek81M-L} representations are obtained for the denominator and residual functions of the Hermite--Pad\'e approximations of system  ${\bf E_{k,\gamma}}=\{E_{\gamma}(\lambda_j z)\}^k_{j=1}$, similar to the well-known Hermite representations~(\cite[chapter 4, \S 2]{NikSor}):	    for $n\geqslant
m_j-1$\,\,\,and\,\,$j=1,2,\ldots,k$
\begin{equation}
\label{eq5.2}	                         
Q_{m}(z;{\bf E_{k,\gamma}})=
\frac{z^{n+m+\gamma}}{\Gamma(n+m+\gamma)}
\int_0^{+\infty} U_{\gamma}(x)
e^{-zx}dx,                                           
\end{equation}
\begin{equation}	  
\label{eq5.3}
R_{n,\overrightarrow{m}}^j(z;{\bf E_{k,\gamma}})=\frac{e^{\lambda_jz}z^{n+m+1}}
{\lambda_j^{\gamma-1}(\gamma)_{n+m}}
\int_0^{\lambda_j} U_{\gamma}(x)
e^{-zx}dx\,,                                              
\end{equation}
where $\Gamma(x)$ --- Euler's gamma function, and
\begin{equation}	  
\label{eq5.4}
U_{\gamma}(x)=x^{n+\gamma-1} \prod_{p=1}^k (x-\lambda_p)^{m_p}.
\end{equation}
Expanding the brackets in~\eqref{eq5.4}, we obtain
$$
U_{\gamma}(x)=x^{n+m+\gamma-1}+\ldots+bx^{n+\gamma-1}, \,\,\,b\neq 0.
$$
Since
$$
\int_0^{+\infty} x^{k+\gamma-1}e^{-zx}dx=\dfrac{\Gamma(k+\gamma)}{z^{k+\gamma}},\,\,\,k=1,2,\ldots,
$$
then
$$
Q_{m}(z;{\bf E_{k,\gamma}})=1+\ldots+b\,\frac{\Gamma(n+\gamma)}{\Gamma(n+m+\gamma)}z^m
$$
--- polynomial of degree $m$ and $Q_{m}(0;{\bf E_{k,\gamma}})=1$. 
Since it is proved in~\cite{Aptek81M-L} that for system ${\bf E_{k,\gamma}}$ for $n\geqslant
m_j-1$\,\,\,($j=1,2,\ldots,k$) problem A has a unique solution (up to a numerical factor), it follows from $Q_{m}(0;{\bf E_{k,\gamma}})=1$ that $H_{n,\overrightarrow m}\neq 0$. Then, according to theorem~\ref{t2}, for  
$(n,\overrightarrow{m})$ under condition $n\geqslant
m_j-1$\,\,\,($j=1,2,\ldots,k$) in some neighborhood of zero
\begin{equation}
\label{eq2.11}
E_{\gamma}(\lambda_j z)-\frac{P^j_{n_j}(z;{\bf E_{k,\gamma}})}{Q_{n,\overrightarrow{m}}(z;{\bf E_{k,\gamma}})}=\sum_{l=n+m+1}^{\infty}\tilde{a}^j_lz^l, \,\,\,j=1,2,\ldots,k.
\end{equation}
It is important for us that, for sufficiently large $n$, such a neighborhood can be taken as an open circle with a center at zero, the radius of which is greater than 1. This statement is a consequence of the main lemma of the work~\cite{Aptek81M-L}, since according to this lemma the zeros of $Q_{n,\overrightarrow{m}}(z;{\bf E_{k,\gamma}})$ for sufficiently large $n$ lie outside the indicated circle. From formula~\eqref{eq5.3} follows the validity of equality
$$
\tilde{a}^j_{n+m+1}= \frac{1}
{\lambda_j^{\gamma-1}(\gamma)_{n+m}}
\int_0^{\lambda_j} U_{\gamma}(x)dx.
$$
We also note that the examples given by M.\,De~Bruyne~\cite{Bruin} show that condition $n\geqslant
m_j-1$, generally speaking, is necessary for the validity of representations \eqref{eq5.2} and \eqref{eq5.3}.

\section{Trigonometric Hermite--Pad\'e and Hermite--Jacobi approximations}\label{s3}

\subsection{Trigonometric Hermite--Pad\'e approximations}\label{s3.1}
In this section we will consider the trigonometric analogue of problem~${\bf A}$. Our goal is to generalize such well-known concepts as partial sums and Pad\'e--Fourier approximations of trigonometric series.

Let ${\bf f^t}=(f^t_1,\ldots,f^t_k)$ be a set of trigonometric series
\begin{equation}
\label{eq3.1}
f^t_j(x)=\frac{a^j_0}{2}+\sum_{l=1}^{\infty}
\left (a^j_l\cos lx+b^j_l\sin lx\right ),\,\,\,
j=1,2,\ldots,k\,                                      
\end{equation}
with real coefficients. We assume that the series~\eqref{eq3.1} converge for all $x\in\mathbb{R}$, and each series defines a function defined on the entire real line. Let us fix an index $n\in
\mathbb{Z}^1_+$ and a multi-index $\overrightarrow{m}=(m_1,\ldots,m_k)$ and consider the following problem:

\vspace{0.1 cm}
{\bf Problem ${\bf A^t}$}. {\it For a set of trigonometric series~\eqref{eq3.1} find a trigonometric polynomial
	$Q^t_m(x)=Q^t_{n,\overrightarrow{m}}(x;{\bf f^t})$, $\deg Q^t_m \leqslant m$ that is identically not equal to zero and such trigonometric polynomials $P^t_j(x)=P^t_{n_j,n,\overrightarrow m}(x;{\bf f^t})$,
	$\deg P^t_j\leqslant n_j$, $n_j=n+m-m_j$, that 
	\begin{equation*}
	R^t_j(x)=R^t_{n_j,n,\overrightarrow{m}}(x;{\bf f^t}):=
	Q^t_m(x)f^t_j(x)-P^t_j(x)=
	\end{equation*}
	\begin{equation}
	\label{eq3.2}
	=\sum_{l=n+m+1}^{\infty}
	\left (\tilde{a}^j_l\cos lx+\tilde{b}^j_l\sin lx\right ),\,\,\,
	j=1,2,\ldots,k\, ,                                  
	\end{equation}
	where $\tilde{a}^j_l$, $\tilde{b}^j_l$, like the coefficients of the trigonometric polynomials $Q^t_m(x)$, $P^t_j(x)$ are, generally speaking, complex numbers.} 

Obviously, polynomials $Q^t_m(x)$,
$P^t_j(x)$ are not uniquely determined by conditions~\eqref{eq3.2}: if the pair $(Q^t_m,P^t)$, where $P^t:=(P^t_1,\ldots, P^t_k)$,
satisfies conditions~\eqref{eq3.2}, then for any non-zero complex number $\lambda$ the new pair $(\lambda
Q^t_m,\lambda P^t)$, where $\lambda P^t:=(\lambda P^t_1,\ldots, \lambda P^t_k)$, also satisfies them. In fact, non-uniqueness may be more significant even in the case $k=1$.
Let us give an example to confirm this.

\begin{example}
	\label{ex2}
	 Let $k=1, \, n=2, \, m=1$, and
$$
f(x)=\sum_{l=1}^{\infty} a_l\cos lx\,\,,
$$
where
$$
a_l=\begin{cases} 
2,&\text{if $l=1,2,4$;}\\
4,&\text{if $l=3$;}\\
\dfrac{1}{l!},&\text{if $l>4$.}
\end{cases} 
$$
Then any solution to the problem ${\bf A^t}$ can be represented in the form: $(\lambda Q^t_1,\lambda P^t_1)$,   $\lambda\in \mathbb{C}$, $\lambda\neq 0$, where polynomials $Q^t_1(x)$, $P^t_1(x)$ in complex form are defined by equalities
$$
Q^t_1(x)=ae^{-ix}-\frac{a+b}{2}  +be^{ix}\,,
$$
$$
P^t_1(x)=\frac{a+3b}{2}e^{-i2x}+\frac{b-a}{2}e^{-ix}+a+b+
\frac{a-b}{2}e^{ix}+\frac{3a+b}{2}e^{i2x}\,,
$$
in which $a$ and $b$ are arbitrary real numbers that are not equal to zero simultaneously, $i=\sqrt{-1}$.
\end{example}

\begin{definition}\label{de6}
	We will say that a problem ${\bf A^t}$ has a unique solution if this solution is unique up to a numerical factor, i.e. for any two solutions
	$(\bar{Q}^t_m,\bar{P}^t)$ and
	$(\bar{\bar{Q}}^t_m,\bar{\bar{P}}^t)$ to the problem ${\bf A^t}$ there is a complex number $\lambda$ such that $(\bar{Q}^t_m,\bar{P}^t)=
	(\lambda \bar{\bar{Q}}^t_m,\lambda \bar{\bar{P}}^t).$
\end{definition}

\begin{definition}\label{de7}
	Let the pair $(Q^t_m,P^t)$, where $P^t=(P^t_1,\ldots,P^t_k)$, be a solution to the problem ${\bf A^t}$. Then the polynomials $Q^t_m(x),\,P^t_1(x),\ldots,\, P^t_k(x)$ and rational fractions
	$$
	\pi^t_j(x;{\bf f^t})=\pi^t_{j,n,\overrightarrow{m}}(x;{\bf f^t})=\frac{P^t_j(x)}{Q^t_m(x)},\,\,\,
	j=1,2,\ldots,k\,
	$$
	will be called, respectively,   {\it trigonometric Hermite--Pad\'e polynomials} and {\it trigonometric Hermite--Pad\'e approximations} ({\it consistent Hermite--Fourier approximations}) for the multi-index $(n,\overrightarrow{m})$ and system ${\bf f^t}$.
\end{definition}

If the problem ${\bf A^t}$ has a unique solution, then the Hermite--Pad\'e trigonometric approximations $\left\{\pi^t_j(x;{\bf f^t})\right\}^k_{j=1} $ are determined uniquely. For $k=1$ a sufficient condition for the uniqueness of the solution to the problem ${\bf A^t}$ is obtained in~\cite{Star-Lab}: for uniqueness, it is sufficient that the determinant $\Delta^1(n,m)\neq 0$ (see below for the definition of $\Delta^1(n,m)$). Under this condition, explicit determinant formulas for polynomials $Q^t_m(x),\,P^t_1(x)$ are found in~\cite{Star-Lab}, similar to the known formulas for representing algebraic Pad\'e polynomials $Q_n(z)$, $P_n(z)$ (~\cite[chapter 1, \S\,1.1]{Bek}). In the case $k=1$ the Hermite--Pad\'e trigonometric approximations we defined are also called {\it trigonometric Pad\'e approximations} or {\it Pad\'e--Fourier approximations}~\cite{Bek, Star-Lab}.

Let us note one important difference between trigonometric $\pi^t_{n,m}(x)$ and algebraic $\pi_{n,m}(z)$ Pad\'e approximations. As already mentioned, for any pair of indices $(n,m)$ the algebraic Pad\'e approximations
$\pi_{n,m}(z)$ are determined uniquely. Trigonometric Pad\'e approximations do not have this property. This conclusion can be made based on example 2. Let us denote by $\pi^{t,1}_{2,1}(x)$ the trigonometric Pad\'e approximation from example 2, which corresponds to the parameters $a=b=1$, and by $\pi^{t,2}_{2,1}(x)$ the trigonometric Pad\'e approximation, which corresponds to the parameters $a=2$, $b=0$. Then we get 
$$
\pi^{t,1}_{2,1}\left (\frac{\pi}{2}\right )=2\neq \frac{-2-6i}{5}=\pi^{t,2}_{2,1}\left(\frac{\pi}{2}\right).
$$

Our immediate goal is to find necessary and sufficient conditions under which the problem ${\bf A^t}$ for the set ${\bf f^t}$ and multi-index $(n,\overrightarrow m)\in \mathbb{Z}^{k+1}_+$ has a unique solution.

Let us write the series~\eqref{eq3.1} and polynomials $Q^t_m(x), P^t_j(x)$ in complex form: 
\begin{equation}
\label{eq3.3}
f^t_j(x)=\sum_{l=-\infty}^{+\infty}c^j_l e^{ilx},\,\,\,    j=1,2,\ldots,k;                            
\end{equation}

\begin{equation}
\label{eq3.4}
Q^t_m(x)=\sum_{p=-m}^{m}u_p e^{ipx}\,, \,\,\,
P^t_j(x)=\sum_{p=-n_j}^{n_j}v^j_p e^{ipx}\,,\,\,\,j=1,2,\ldots,k,
\end{equation}
where 
$
u_p, v^j_p\in \mathbb{C},\,\,\,c^j_0=\dfrac{a^j_0}{2},\,\, c_l^j=\dfrac{a_l^j-ib_l^j}{2},\,\,
c^{j}_{-l}=\overline{c}^{\,j}_{l},\,\,j=1,2,\ldots,k;\,\, l=1,2,\ldots\,.
$
Then equalities~\eqref{eq3.2} will take the form
\begin{equation}
\label{eq3.5}
R^t_j(x)=\sum_{l=n+m+1}^{+\infty}\left (\tilde{c}^j_l e^{ilx}+\tilde{c}^j_{-l} e^{-ilx}\right ),  \,\,\,j=1,2,\ldots,k.
\end{equation}
Let us introduce into consideration matrices and determinants, the elements of which are the coefficients of the trigonometric series $f^t_j(x)$ of system ${\bf f^t}=\{f^t_1,\ldots,f^t_k\}$.

We assign each $l\in \mathbb{Z}$ to a matrix-row
$$
\mathbb{C}^j_l=\left (c^j_{l+m}\,\,\,c^j_{l+m-1}\,\,\ldots\,\,c^j_{l+1}\,\,c^j_{l}\,\,\,c^j_{l-1}\,\,\ldots\,\,c^j_{l-m+1}\,\,c^j_{l-m}\right ),\,\,j=1,2,\ldots,k,
$$
and a real number $x$ to a row matrix
$$
E^t_m(x)=\left (e^{-imx}\,\,e^{-i(m-1)x}\,\,...\,\,e^{-ix}\,\,1\,\,\,e^{ix}\,\,...\,\,e^{i(m-1)x}\,\,e^{imx}\right )\,.
$$
For a given $j\in \{1,2,\ldots,k\}$,
a fixed index $n\in
\mathbb{Z}^1_+$ and a non-zero multi-index $\overrightarrow{m}=(m_1,\ldots,
m_k)$ assuming that
$m_j\neq 0$, we define matrices of order $m_j\times(2m+1)$
\begin{equation*}
F^j_+:=\left[ \! \begin{array}{cccc}
\mathbb{C}^j_{n_j+m_j}\\
\mathbb{C}^j_{n_j+m_j-1} \\ 
\vdots\\
\mathbb{C}^j_{n_j+1}
\end{array} \!\, \right]=
\left( \! \begin{array}{cccc}
c_{n_j+m+m_j}^{j} & c_{n_j+m+m_j-1}^{j} & \ldots & c_{n_j-m+m_j}^{j} \\
c_{n_j+m+m_j-1}^{j} & c_{n_j+m+m_j-2}^{j} & \ldots & c_{n_j-m+m_j-1}^{j} \\
\ldots & \ldots & \ldots & \ldots \\
c_{n_j+m+1}^{j} & c_{n_j+m}^{j} & \ldots & c_{n_j-m+1}^{j}\\
\end{array} \! \right),
\end{equation*}

\begin{equation*}
F^j_-:=\left[ \! \begin{array}{cccc}
\mathbb{C}^j_{-n_j-1}\\
\mathbb{C}^j_{-n_j-2} \\ 
\vdots\\
\mathbb{C}^j_{-n_j-m_j}
\end{array} \!\, \right]=
\left( \! \begin{array}{cccc}
c_{-n_j+m-1}^{j} & c_{-n_j+m-2}^{j} & \ldots & c_{-n_j-m-1}^{j} \\
c_{-n_j+m-2}^{j} & c_{-n_j+m-3}^{j} & \ldots & c_{-n_j-m-2}^{j} \\
\ldots & \ldots & \ldots & \ldots \\
c_{-n_j+m-m_j}^{j} & c_{-n_j+m-m_j-1}^{j} & \ldots & c_{-n_j-m-m_j}^{j}\\
\end{array} \! \right).
\end{equation*}
Let us introduce into consideration a determinant of order $2m+1$
$$
D(n, \overrightarrow{m}; x):=\det \left[ \! \begin{array}{cccccccccc}
F^k_+  & \ldots & F^{2}_+ & F^1_+ & E^t_m(x) & F^1_-  & F^2_-& \ldots & F^k_-
\end{array} \!\, \right]^{T}=
$$
$$
:=\det \left[  \begin{array}{ccccccc}
F^k_+\\ 
\vdots\\
F^1_+ \\
E^t_m(x)\\
F^1_-\\
\vdots\\
F^k_-
\end{array}  \right].
$$
If $m_j=0$, we assume that the determinant
$D(n, \overrightarrow{m}; x)$ does not contain block matrices $F^j_{\pm}$. Let $H^t_{n,\,\overrightarrow{m}}$ denote the matrix of order $2m\times(2m+1)$, obtained from the elements of the determinant $D(n, \overrightarrow{m}; x)$ after removing the $(m+1)$-th row $E^t_m(x)$ from it. If in the determinant $D(n, \overrightarrow{m}; x)$ the row $E^t_m(x)$ is replaced by the row $\mathbb{C}^j_{l}$, we obtain a new determinant $d^j_l(n,\overrightarrow{m})$. Let $\Delta^k(n,\overrightarrow{m})$ denote the determinant of order $2m$, obtained by deleting the $(m+1)$-th row and $(m+1)$-th column in the determinant $D(n, \overrightarrow{m}; x)$.

\begin{definition}\label{de8}
	A multi-index $(n,\overrightarrow{ m})\in \mathbb{Z}^{k+1}_+$, $\overrightarrow{m}\neq (0,\ldots,0)$
	will be called {\it weakly normal} for a system $\bf f^t$,
	if $H^t_{n,\overrightarrow{m}}$ is a matrix of full rank, i.e. $rank\,
	H^t_{n,\,\overrightarrow{m}}=2m$.
\end{definition}

\begin{definition}\label{de9}
	A system $\bf f^t$ is called
	{\it weakly perfect}, if all multi-indices
	$(n,\overrightarrow{m})\in \mathbb{Z}^{k+1}_+$, $\overrightarrow{m}\neq (0,\ldots,0)$ are weakly normal for $\bf f^t$.	
\end{definition}

\begin{theorem}\label{t3}
	In order for a problem $\bf A^t$ to have a unique solution for a fixed multi-index $(n,\overrightarrow{m})$, $\overrightarrow{m}\neq (0,\ldots,0)$ and a system $\bf f^t$, it is necessary and sufficient that the multi-index $(n,\overrightarrow{m})$ be weakly normal for $\bf f^t$, i.e. $rank\, H^t_{n,\,\overrightarrow{m}}=2m$.
	
	If $rank\,
	H^t_{n,\,\overrightarrow{m}}=2m$, then for a certain choice of the normalizing factor the following representations are valid for solutions of problem $\bf A^t$: for $j=1,2,\ldots,k$
	\begin{equation}
	\label{eq3.6}
	Q^t_m(x)=D(n, \overrightarrow{m}; x)\,,             
	\end{equation}
	\begin{equation}
	\label{eq3.7}
	P_j^t(x)=
	\sum_{p=-n_j}^{n_j}d_{p}^j(n,\overrightarrow m)e^{ipx}\,,                         
	\end{equation}
	\begin{equation}
	\label{eq3.8}	
	R^t_j(x)=\sum_{p=n+m+1}^{\infty}\bigl (\,\, d_{p}^j(n,\overrightarrow m)e^{ipx}+d_{-p}^j(n,\overrightarrow m)e^{-ipx}\,\,\bigl )\,.                                    
	\end{equation}
\end{theorem}

\begin{proof}
	Let the desired polynomial $Q^t_m(x)$ have the form~\eqref{eq3.4}.
	After transformations we obtain
	$$
	Q^t_m(x)f^t_j(x)=\sum_{l=-\infty}^{\infty}\left  (\sum_{p=-m}^{m} c^j_{l-p} u_p\right )e^{ilx}=\sum_{l=-\infty}^{\infty}\tilde{c}^j_{l}e^{ilx}\,,
	$$
	where 
	\begin{equation}
	\label{eq3.9}	
	\tilde{c}^j_{l}=\sum_{p=-m}^{m}  c^j_{l-p}u_p\,,\,\,\, l\in \mathbb{Z}\,.                                    
	\end{equation}
	We choose the coefficients $u_p$, $p=\overline{-m,m}$ of the polynomial $Q^t_m(x)$ so that
	$$
	\tilde{c}^j_{l}=0\,,\,\,\,l=\pm(n_j+1),\ldots,\pm(n_j+m_j),\,\,\,j=1,2,\ldots,k\,,
	$$
	and we set
	$$
	P^t_j(x)=\sum_{p=-n_j}^{n_j}\tilde{c}^j_{p} e^{ipx}\,.
	$$
	It is obvious that the polynomials $Q^t_m(x)$, $P^t_j(x)$ chosen in this way satisfy conditions~\eqref{eq3.2} and~\eqref{eq3.5}. It remains to investigate the compatibility of the system of equations
	\begin{equation}
	\label{eq3.10}
	\sum_{p=-m}^{m}  c^j_{l-p}u_p=0\,,\,\,\,l=\pm(n_j+1),\ldots,\pm(n_j+m_j)\,,\,\,\,
	j=1,2,\ldots,k.      	\end{equation}
	
	System~\eqref{eq3.10} can be written in matrix form
	$$
	H^t_{n,\overrightarrow{m}} \cdot u^T=\theta^T\,,
	$$
	where $u=(u_{-m}\,\ldots\,u_{-1}\,\,u_0\,\,u_1\,\,\ldots\,u_m)$ is a row matrix of unknown coefficients, and
	$\theta $ is a row matrix of order $1 \times(2m+1)$, all elements of which are zero. Since system~\eqref{eq3.10} is homogeneous and the number of unknowns $2m+1$ in it is one greater than the number of equations $2m$, it follows from the Kronecker--Capelli theorem that system~\eqref{eq3.10} has a non-zero solution. Moreover, the set of all linearly independent solutions of system~\eqref{eq3.10}
	consists of one fundamental solution if and only if $rank\, H^t_{n,\,\overrightarrow{m}}=2m$. In this case, all other non-zero solutions are obtained by multiplying this fundamental solution by the number $\lambda\neq0$. Thus, the first part of theorem~\ref{t3} is proven.
	
	Let us now prove equalities~\eqref{eq3.6}--\eqref{eq3.8}. Since, by assumption, the rank of the matrix
	$H^t_{n,\,\overrightarrow{m}}$ is equal to $2m$, then for some $s \in
	\{1,\ldots,2m+1\}$ the determinant obtained from the elements of the matrix $H^t_{n,\,\overrightarrow{m}}$ as a result of deleting the $s$-th
	column in it is different from zero. For definiteness, let us assume that
	$s=m+1$. Then, having fixed the unknown $u_0$, we obtain a square inhomogeneous system
	\begin{equation}
	\label{eq3.11}
	\sum_{p=-m}^{-1}  c^j_{l-p}u_p+\sum_{p=1}^{m} c^j_{l-p}u_p=- c^j_{l}u_0\,,\,\,\,
	l=\pm(n_j+1),\ldots,\pm(n_j+m_j)\,,\,\,\,j=1,2,\ldots,k
	\end{equation}
	the main determinant of which $\Delta^k(n,\overrightarrow{m})\neq 0$. Note that
	$u_0\neq 0$. Otherwise, system~\eqref{eq3.11}, and therefore system~\eqref{eq3.10} would have only zero solutions. System~\eqref{eq3.11} has a unique non-zero solution, and it can be found using Cramer's formulas:
	$$
	u_p=\frac{\Delta^k_p(n,\overrightarrow{m})}{\Delta^k(n,\overrightarrow{m})},\,\,\,p=\overline{-m,m},\,\,p\neq 0,
	$$
	where $\Delta^k_p(n,\overrightarrow{m})$ is the determinant obtained from the determinant $\Delta^k(n,\overrightarrow{m})$ by replacing the $p$-th column with a column of free terms. If we put $\Delta^k_0(n,\overrightarrow{m}):=u_0\Delta^k(n,\overrightarrow{m})$, then
	\begin{equation}
	\label{eq3.12}
	Q^t_m(x)=\sum_{p=-m}^{m}u_p e^{ipx}=
	\sum_{p=-m}^{m}\dfrac{\Delta^k_p(n,\overrightarrow{m})}{\Delta^k(n,\overrightarrow{m})} \,e^{ipx}.    
	\end{equation}
	Expanding the determinant $D(n, \overrightarrow{m}; x)$ by the elements of the $(m+1)$-th row and comparing with~\eqref{eq3.12}, we conclude that
	\begin{equation}
	\label{eq3.13}
	Q^t_m(x)=u_0\frac{D(n, \overrightarrow{m}; x)}{\Delta^k(n,\overrightarrow{m})}.   
	\end{equation}
	Comparing~\eqref{eq3.9} and~\eqref{eq3.12}, we notice that to find $\tilde{c}^j_{p}$ we only need to replace $e^{ipx}$ with 
	$c^j_{l-p}$ in ~\eqref{eq3.12}. Taking into account the introduced notations, we obtain that
	$$
	\tilde{c}^j_{p}=u_0\frac{d^j_p(n, \overrightarrow{m}; x)}{\Delta^k(n,\overrightarrow{m})}.
	$$
	Therefore, the polynomial $P^t_j(x)$ and the remainder term $R^t_j(x)$ can be represented as
	\begin{equation}
	\label{eq3.14}
	P_j^t(x)=\frac{u_0}{\Delta^k(n,\overrightarrow{m})}
	\sum_{p=-n_j}^{n_j}d_{p}^j(n,\overrightarrow m)\,e^{ipx}\,,                         
	\end{equation}
	\begin{equation}
	\label{eq3.15}
	R_j^t(x)=\frac{u_0}{\Delta^k(n,\overrightarrow{m})}
	\sum_{p=n+m+1}^{\infty}\bigl (\,d_{p}^j(n,\overrightarrow m)\,e^{ipx}+d_{-p}^j(n,\overrightarrow m)\,e^{-ipx}\,\bigl )\,.                         
	\end{equation}
	Multiplying equalities~\eqref{eq3.13}--\eqref{eq3.15} by the normalizing factor $\Delta^k(n,\overrightarrow{m})/u_0$, we obtain~\eqref{eq3.6}--\eqref{eq3.8}.
	
	If, when deleting a column with number $s\neq m+1$ in the matrix $H^t_{n,\,\overrightarrow{m}}$ we obtain a determinant different from zero, then, reasoning similarly, we arrive at representations~\eqref{eq3.6}--\eqref{eq3.8}. 
	Now it remains to note that since
	$rank\,
	H^t_{n,\,\overrightarrow{m}}=2m$, the determinant $D(n, \overrightarrow{m}; x)$ cannot be identically equal to zero. Therefore, theorem~\ref{t3} is proven. 
\end{proof}

 \subsection{Remarks and corollaries}\label{s3.2}
 \begin{remark}\label{re2}
 	From representation~\eqref{eq3.6} for the polynomial $Q^t_m(x)$ it follows that the component $m_j$ of the multi-index $\overrightarrow{m}$ determines the number of coefficients of the trigonometric series $f^t_j(x)$, that are taken into account when constructing the polynomial
 	$Q^t_m(x)$. In particular, if $m_j=0$, then the determinant
 	$D(n, \overrightarrow{m}; x)$ does not contain blocks $F^j_{\pm}$ and, therefore, when constructing the polynomial
 	$Q^t_m(x)$, the trigonometric series $f^t_j(x)$ is not taken into account, and the order of the multi-index $\overrightarrow{m}$ is determined by the remaining non-zero components.
 	For example, if $\overrightarrow{m}=(m_1,0,\ldots,0)$, then $m=m_1$ and then, as in the case of $k=1$,
 	when finding $Q^t_m(x)$ only the coefficients of the series
 	$f^t_1(x)$ are taken into account. For such a multi-index, formula~\eqref{eq3.6} coincides with the formula for the denominator of the trigonometric Pad\'e approximation $\pi^t_{n,m}(x)$ of the function $f^t_1(x)$, which is obtained in~\cite{Star-Lab} under the more restrictive condition $\Delta^1(n,\overrightarrow{m})\neq 0$. 
 \end{remark}
 
 \begin{remark}\label{re3} In theorem~\ref{t3} it is assumed that the multi-index $\overrightarrow{m}$ is non-zero. If $\overrightarrow{m}=(0,\ldots,0)$, then the solution to problem $\bf A^t$ is obvious: up to a numerical factor $Q^t_m(x)\equiv 1$, and the polynomial $P^t_j(x)$ coincides with the $n$-th partial sum of the series $f^t_j(x)$.
 \end{remark}
 
 \begin{remark}\label{re4}
 	In proving theorem~\ref{t3} our assumption about the convergence of series~\eqref{eq3.1} was not taken into account. Therefore, all the statements of the theorem remain valid if series~\eqref{eq3.1} are formal. 
 \end{remark}
 
 It should also be said that if the index $(n,\overrightarrow{m})$ is not weakly normal for ${\bf f^t}$, then the polynomials
 $Q^t_m(x), P^t_1(x),\ldots,P^t_k(x)$, given by formulas~\eqref{eq3.6} and~\eqref{eq3.7}, are not solutions of problem $\bf A^t$. In particular, in example 2 the index (2,\,1) is not weakly normal, and if we calculate, for example, the polynomial $Q^t_1(x)$
 using formula~\eqref{eq3.6}, we get $Q^t_1(x)\equiv 0$. This example also shows that for a multi-index $(n,\overrightarrow{m})$ that is not weakly normal, the coefficients of the polynomials
 $Q^t_m(x)$, $P^t_j(x)$ can be complex numbers. For a weakly normal multi-index this is not the case.

\begin{corollary}
	 Let a multi-index $(n,\overrightarrow{m})$ be weakly normal for a system ${\bf f^t}$. Then the polynomials $Q^t_m(x), P^t_1(x),\ldots,P^t_k(x)$ --- solutions of problem ${\bf A^t}$, are real trigonometric polynomials.
	 \end{corollary}
 
 \begin{proof}
 	By assumption, the coefficients of trigonometric series~\eqref{eq3.1} are real numbers. Therefore, the following equalities are valid: $c^{j}_{-p}=\overline{c}^{\,j}_{p}$,\,\,$j=1,2,\ldots,k;$\,\, $p=1,2,\ldots$\,\,. In this case we get that
 	$$
 	\overline{D(n, \overrightarrow{m}; x)}=D(n, \overrightarrow{m}; x),\,\,\,\overline{d_{p}^j(n,\overrightarrow m)}=d_{-p}^j(n,\overrightarrow m).
 	$$
 	To verify this, it is enough to swap the rows and columns of the corresponding determinants that are equidistant from the edges. From these equalities and formulas~\eqref{eq3.6},~\eqref{eq3.7} follows the statement of corollary 2.
 \end{proof}

 \begin{corollary}
 	In order for problem $\bf A^t$ to have a unique solution for any multi-index $(n,\overrightarrow{m})$,
 	it is necessary and sufficient that the system $\bf f^t$ be weakly perfect.
 \end{corollary}
 
 \subsection{Trigonometric Hermite--Jacobi approximations}\label{s3.3}
 Let us introduce into consideration the trigonometric Hermite--Jacobi approximations.
 
 \begin{definition}\label{de10}
 	Rational functions of the form
 	$$
 	\widehat{\pi}^t_j(x;{\bf f^t})=	\widehat{\pi}^t_{j,n,\overrightarrow{m}}(x;{\bf f^t})=\frac{\widehat{P}^t_{j}(x)}{\widehat{Q}^t_m(x)},\,\,\,
 	j=1,2,\ldots,k\,,
 	$$
 	where $\widehat{Q}^t_m(x)=\widehat{Q}^t_{n,\overrightarrow{m}}(x;{\bf f^t})$,
 	$\widehat{P}^t_{j}(x)=\widehat{P}_{n_j,n,\overrightarrow{m}}^t(x;{\bf f^t})$ --- trigonometric polynomials whose degrees are not higher than $m$ and $n_j$, respectively, $n_j=n+m-m_j$, will be called {\it trigonometric Hermite--Jacobi approximations} for a multi-index $(n,\overrightarrow{m})$ and a system of functions ${\bf f^t}$, if each $\widehat{\pi}^t_j(x;{\bf f^t})$ can be represented by a trigonometric series and for $j=1,2,\ldots,k$
 	\begin{equation}
 	\label{eq3.16}
 	f^t_j(x)-\frac{\widehat{P}^t_{j}(x)}{\widehat{Q}^t_m(x)}=\sum_{l=n+m+1}^{\infty}
 	(\tilde{a}^j_l\cos lx+\tilde{b}^j_l\sin lx ).
 	\end{equation}
 \end{definition}
 Polynomials $\widehat{Q}^t_m(x), \widehat{P}^t_{1}(x),\ldots,\widehat{P}^t_{k}(x)$, satisfying conditions~\eqref{eq3.16}, will be called {\it trigonometric Hermite--Jacobi polynomials} for the multi-index $(n,\overrightarrow{m})$ and the system ${\bf f^t}$. 
 
 Unlike the trigonometric Hermite--Pad\'e approximations, the trigonometric Hermite--Jacobi approximations may not exist, moreover, if they do exist, they may not coincide with the trigonometric Hermite--Pad\'e approximations. In particular, for $k=1$ for the pair of indices $(6, 6)$ the trigonometric Hermite--Jacobi approximations do not exist (\cite{Star-Lab}) for the non-differentiable Weierstrass function                                                                $$
 f^t(x)=\sum_{l=0}^{\infty}q^l\cos (2p+1)^lx,\,\,\,p\in \mathbb{N},\,\,0<q<1.
 $$
 In work
 \cite{Nemeth} the Gibbs effect for Pad\'e approximations of the sign function $sgn\,x$ is investigated and, in particular, for $k=1$ explicit expressions are found for fractions $\pi^t_{n,\,m}(x;s):=\pi^t_{1,n,\,m}(x;s)$,
 $\widehat{\pi}^{t}_{n,\,m}(x; s):=\widehat{\pi}^{t}_{1,n,\,m}(x; s)$, where $s(x)=sgn(\cos x)$, from which it follows that these fractions are different for all $m\geqslant 1$. 
 In the case $k=1$ other examples of functions for which trigonometric Hermite--Jacobi approximations do not exist, as well as examples of functions for which trigonometric Hermite--Jacobi approximations exist but do not coincide with trigonometric Hermite--Pad\'e approximations, can be constructed based on the results of works~\cite{Suet2, Suet1, SuetDis}.  
 
   In the work~\cite{StarKechOsn}, based on the representations of Ch.\,Hermite of the Hermite--Pad\'e polynomials of exponential functions, for an arbitrary $k\geqslant1$, examples of systems of functions are constructed for which the trigonometric Hermite--Jacobi approximations exist, but do not coincide with the trigonometric Hermite--Pad\'e approximations. The results of work~\cite{StarKechOsn} can be significantly supplemented if instead of exponential functions we take Mittag-Leffler functions.
   
   Let us consider families of trigonometric series
   \begin{equation*}
   G_{\gamma}(x;\lambda)=\sum_{l=0}^{\infty}\dfrac{\lambda^l}
   {(\gamma)_l}\cos lx\,, 
   \end{equation*}
   depending on the parameters $\lambda\in \mathbb{R}$ and $\gamma\in\mathbb{R}\setminus \mathbb{Z_{-}}$. It is easy to see that
   $G_{\gamma}(x;\lambda)$ is the real part of the function
   $E_{\gamma}(e^{i \lambda x})$.
   Considering that $\{\lambda_j\}_{j=1}^k$ are different non-zero real numbers, we define a new system of functions ${\bf G_{k,\gamma}}=\{G_{\gamma}(x;\lambda_j)\}_1^k$.

   \begin{theorem}\label{t8}
   	For a multi-index
   	$(n,\overrightarrow{m})\in \mathbb{Z}^{k+1}_+$, $\overrightarrow{m}\neq (0,\ldots,0)$, satisfying the condition $n\geqslant m_j$ $(j=1,2,\ldots,k)$, for all $n$, starting from some $n_0$, there exist trigonometric Hermite--Jacobi approximations $\{{\widehat{\pi}}^t_j(x;{\bf G_{k,\gamma}})\}_{j=1}^k$, and with the appropriate normalization their denominator can be represented in the form: 
   	\begin{equation*}
   	\widehat{Q}^t_m(x;{\bf G_{k,\gamma}})=
   	{Q}_m(e^{ix};{\bf E_{k,\gamma}})\,\cdotp\, \overline{{Q}_m(e^{ix};{\bf E_{k,\gamma}})}=
   	\end{equation*}	
   	\begin{equation*}
   	=\frac{1}{\Gamma^2(n+m+\gamma)}\left |\int\limits_0^{+\infty} U_{\gamma}(t) e^{-t(\cos x+i\sin x)} dt\right |^2. 
   	\end{equation*}   	   	
   	If, in addition, $0<\lambda_1<\lambda_2<\ldots<\lambda_k$, then among $\{{\widehat{\pi}}^t_j(x;{\bf G_{k,\gamma}})\}_{j=1}^k$ there is a trigonometric Hermite--Jacobi approximation $\widehat{\pi}^t_j(x;{\bf G_{k,\gamma}})$ that is not a trigonometric Hermite--Pad\'e approximation ${\pi}^t_j(x;{\bf G_{k,\gamma}})$.
   \end{theorem}
 \begin{proof}
 	The polynomial $Q_m(z;{\bf  E_{k,\gamma}})$ is represented by formula~\eqref{eq5.2}, and in the circle $|z|\leqslant 1$ equalities~\eqref{eq2.11} are valid. In~\eqref{eq2.11} we set $z=e^{ix}$, and then equate the real parts of the expressions to the left and right of the new equality sign. As a result we get that
 	\begin{equation}
 	\label{eq3.21}
 	G_{\gamma}(x;\lambda_j)-Re\left \{\dfrac{P^j_{n_j}(e^{ix};{\bf E_{k,\gamma}})}{Q_m(e^{ix};{\bf E_{k,\gamma}})} \right \} =\sum_{l=n+m+1}^{\infty}
 	\tilde{a}^j_l\cos lx. 
 	\end{equation}
 	It remains to show that for $n\geqslant \max\{ m_j: 1\leqslant j\leqslant k\}$
 	\begin{equation}
 	\label{eq3.22}
 	\widehat{\pi}^t_j(x;{\bf G_{k,\gamma}})=Re\left \{\dfrac{P^j_{n_j}(e^{ix};{\bf E_{k,\gamma}})}{Q_m(e^{ix};{\bf E_{k,\gamma}})} \right \}.
 	\end{equation}
 	
 	To do this, we represent the polynomials $Q_m(z;{\bf E_{k,\gamma}})$, $P^j_{n_j}(z;{\bf E_{k,\gamma}})$ in the form
 	$$
 	Q_m(z;{\bf E_{k,\gamma}})=\sum_{p=0}^m b_pz^p=1+...+b\,\frac{\Gamma(n+\gamma)}{\Gamma(n+m+\gamma)}z^m,\,\,\,P^j_{n_j}(z;{\bf E_{k,\gamma}})=\sum_{p=0}^{n_j} a^j_pz^p.
 	$$	
 	Since $ b_p$ and $a^j_p$ are real numbers, then for $z=e^{ix}$
 	$$
 	Re\left \{\pi^j_{n_j,n,\overrightarrow{m}}(e^{ix};{\bf E_{k,\gamma}})\right \}
 	=\frac{1}{2}\,\left(\frac{P^j_{n_j}(z)}{Q_m(z)}+
 	\frac{\overline{P^j_{n_j}(z)}}{\overline{Q_m(z)}}\right)=
 	$$	
 	\begin{equation}
 	\label{eq3.31}
 	=\frac{1}{2}\frac{\sum^{n_j}_{p=0} a^j_p\,e^{ikx}\cdot \sum^m_{l=0}
 		b_l\,e^{-ilx}+\sum^{n_j}_{p=0} a_p\,e^{-ipx}\cdot \sum^m_{s=0}
 		b_s\,e^{isx}}{\sum^m_{s=0} b_s\,e^{isx}\cdot \sum^m_{p=0}
 		b_p\,e^{-ipx}}\,.
 		\end{equation}
 	Notice, that
 	$$
 	\sum^m_{s=0} b_s\,e^{isx}\cdot \sum^m_{l=0}
 	b_l\,e^{-ilx}=\sum^m_{s=0}\sum^m_{l=0}b_s\,b_l\cos (s-l)x\,,
 	$$
 	$$
 	\sum^{n_j}_{p=0} a^j_p\,e^{ikx}\cdot \sum^m_{l=0}
 	b_l\,e^{-ilx}+\sum^{n_j}_{p=0} a_p\,e^{-ipx}\cdot \sum^m_{s=0}
 	b_s\,e^{isx}
 	=2\sum^{n_j}_{p=0}\sum^m_{l=0}a^j_p\,b_l\cos (p-l)x\,.
 	$$
 	Therefore, for $j=1,2,...\,,k$
 	\begin{equation}
 	\label{eq3.23}
 	Re\left \{\pi^j_{n_j,n,\overrightarrow{m}}(e^{ix};{\bf E_{k,\gamma}})\right \}
 	=
 	\frac{\sum^{n_j}_{p=0}\sum^m_{l=0}a^j_p\,b_l\cos
 		(p-l)x}{\sum^m_{s=0}\sum^m_{l=0}b_s\,b_l\cos (s-l)x}=:\dfrac{\widetilde{P}^t_j(x)}{\widetilde{Q}^t_m(x)}.
 	\end{equation}
 	From the condition $n\geqslant m_j$ $(j=1,2,\ldots,k)$ on the multi-index $(n,\overrightarrow{m})$ it follows that $n_j\geqslant m$. Therefore, taking into account~\eqref{eq3.23}, we obtain that $\deg \widetilde{Q}^t_m\leqslant m$, $\deg \widetilde{P}^t_j\leqslant n_j$, $j=1,2,\ldots,k$. Therefore, identity~\eqref{eq3.22} and the first part of theorem~\ref{t8} are proved.
 	
 	Let us prove the second part of the theorem.
 	From~\eqref{eq3.21} and~\eqref{eq3.23}
 	it follows that
 	$$
 	G_{\gamma}(x;\lambda_j)-{\widehat{\pi}}^t_j(x;{\bf G_{k,\gamma}})=	G_{\gamma}(x;\lambda_j)-\dfrac{\widetilde{P}^t_j(x)}{\widetilde{Q}^t_m(x)}=\sum_{l=n+m+1}^{\infty}
 	\tilde{a}^j_l\cos lx. 
 	$$
 	Using the trigonometric formula for the product of two cosines, from the last equality and from~\eqref{eq5.3} and~\eqref{eq3.23} we obtain
 	\begin{equation}
 	\label{eq3.24}
 	\widetilde{Q}^t_m(x)\,G(x;\lambda_{j})-\widetilde{P}^t_{j}(x)
 	=\alpha^j_{n+1}\cos (n+1)x+\alpha^j_{n+2}\cos (n+2)x+\ldots\,,
 	\end{equation}
 	where $\alpha^j_{n+1}=2b_mb_0\,\tilde{a}^j_{n+m+1}$. Let us recall that
 	$$
 b_0=1,\,\,\,	b_m=b\,\frac{\Gamma(n+\gamma)}{\Gamma(n+m+\gamma)},\,\,\, \tilde{a}^j_{n+m+1}= \frac{1}
 	{\lambda_j^{\gamma-1}(\gamma)_{n+m}}
 	\int_0^{\lambda_j} U_{\gamma}(x)dx.
 	$$
 	 	If $0<\lambda_1<\lambda_2<\ldots<\lambda_k$, then it is easy to show that, for example, $\tilde{a}^1_{n+m+1}\neq 0$, and then in~\eqref{eq3.24} $\alpha^1_{n+1}\neq 0$. In this case, from the definition of the trigonometric Hermite--Pad\'e approximations it follows that $\widehat{\pi}^t_1(x;{\bf G_{k,\gamma}})\neq {\pi}^t_1(x;{\bf G_{k,\gamma}}).$
 	
 \end{proof}
 
 \section{Linear and nonlinear Hermite--Chebyshev approximations}\label{s4}
 
 In this paragraph, the terminology is partially borrowed from work~\cite{Suet2}.
 
 \subsection{Linear Hermite--Chebyshev approximations}\label{s4.1}
 Let the set ${\bf f^{ch}}=(f^{ch}_1,\ldots,f^{ch}_k)$ consist of functions represented by Fourier series of Chebyshev polynomials
 $T_n(x)=\cos(n\arccos x)$	
 \begin{equation}
 \label{eq4.1}	
 f^{ch}_j(x)=\frac{a^j_0}{2}+\sum_{l=1}^{\infty}
 a^j_lT_l(x),\,\,\,
 j=1,2,\ldots,k\,                                      
 \end{equation}
 with real coefficients. Let us fix an index $n\in
 \mathbb{Z}^1_+$ and a multi-index $\overrightarrow{m}=(m_1,\ldots,m_k)\in \mathbb{Z}^k_+$ and consider the following problem.
 
 \vspace{0.1 cm}
 {\bf Problem ${\bf A^{ch}}$}. {\it For a system of functions ${\bf f^{ch}}$ find a polynomial
 	$
 	Q^{ch}_m(x)=Q^{ch}_{n,\overrightarrow{m}}(x;{\bf f^{ch}})=\sum_{p=0}^{m}u_pT_p(x)
 	$ 
 	that is identically not equal to zero and such polynomials
 	$
 	P^{ch}_j(x)=P^{ch}_{n_j,n,\overrightarrow m}(x;{\bf f^{ch}})=\sum_{p=0}^{n_j}v^j_p\,T_p(x)
 	$, $n_j=n+m-m_j$, 
 	that for $j=1,2,\ldots,k$
 	\begin{equation}
 	\label{eq4.2}
 	Q^{ch}_m(x)f^{ch}_j(x)-P^{ch}_j(x)
 	=\sum_{l=n+m+1}^{\infty}\tilde{a}^j_l\,T_l(x)\,,                                
 	\end{equation}
 	where $\tilde{a}^j_l$ and the coefficients of the polynomials $Q^{ch}_m(x)$, $P^{ch}_j(x)$ are, generally speaking, complex numbers.}

 \begin{definition}\label{de11}
 	If the pair $(Q^{ch}_m,P^{ch})$, where $P^{ch}=(P^{ch}_1,\ldots,P^{ch}_k)$, is a solution to problem ${\bf A^{ch}}$, then the polynomials $Q^{ch}_m(x),\,P^{ch}_1(x),\ldots,P^{ch}_k(x)$ and rational fractions
 	$$
 	\pi^{ch}_j(x;{\bf f^{ch}})=\pi^{ch}_{n_j,n,\overrightarrow{m}}(x;{\bf f^{ch}})=\frac{P^{ch}_j(x)}{Q^{ch}_m(x)},\,\,\,
 	j=1,2,\ldots,k
 	$$
 	will be called, respectively, {\it Hermite--Chebyshev polynomials} and {\it linear Hermite--Chebyshev approximations} for the multi-index  $(n,\overrightarrow{m})$ and the system ${\bf f^{ch}}$.
 \end{definition}
 
 The solution to problem ${\bf A^{ch}}$ can be obtained based on theorem~\ref{t3}. To do this, in~\eqref{eq4.1} we replace $x$ with $\cos x$. Then, putting $\tilde{f}_j^t(x)=f_j^{ch}(\cos x)$, we obtain a set of trigonometric series ${\bf \tilde{f}^t}=(\tilde{f}^t_1,\ldots,\tilde{f}^t_k)$ of the form
 \begin{equation*}
 \label{eq4.3}
 \tilde{f}^t_j(x)=\frac{a^j_0}{2}+\sum_{l=1}^{\infty}
 a^j_l\cos lx,\,\,\,
 j=1,2,\ldots,k.                                    
 \end{equation*}
 For a system ${\bf \tilde{f}^t}$ problem ${\bf A^{t}}$ has a unique solution only if the multi-index $(n,\overrightarrow{m})\in \mathbb{Z}^{k+1}_+$ is weakly normal for $\bf \tilde{f}^t$. For a weakly normal index, the trigonometric Hermite--Pad\'e polynomials, with an appropriate choice of the normalizing factor, can be represented in the form
 \begin{equation}
 \label{eq4.4}
 Q^t_m(x;{\bf \tilde{f}^t})=D(n, \overrightarrow{m}; x)\,,            
 \end{equation}
 \begin{equation}
 \label{eq4.5}
 P_j^t(x;{\bf \tilde{f}^t})=
 \sum_{p=-n_j}^{n_j}d_{p}^j(n,\overrightarrow m)e^{ipx}\,                          
 \end{equation}
 and for them
 \begin{equation}
 \label{eq4.6}	
 Q^t_m(x;{\bf \tilde{f}^t})\tilde{f}^t_j(x)-P^t_j(x;{\bf \tilde{f}^t})=\sum_{p=n+m+1}^{\infty}\bigl (\,\, d_{p}^j(n,\overrightarrow m)e^{ipx}+d_{-p}^j(n,\overrightarrow m)e^{-ipx}\,\,\bigl )\,.                                    
 \end{equation}
 Since for the system $\bf \tilde{f}^t$ the coefficients $b_l^j=0$ for $l=1,2,\ldots;\,\,\,j=1,2,\ldots,k$, the block matrices $F^j_{\pm}$ for this system have the form:
 \begin{equation*}
 F^j_+=
 \left( \! \begin{array}{cccc}
 h_{n_j+m+m_j}^{j} & h_{n_j+m+m_j-1}^{j} & \ldots & h_{n_j-m+m_j}^{j} \\
 h_{n_j+m+m_j-1}^{j} & h_{n_j+m+m_j-2}^{j} & \ldots & h_{n_j-m+m_j-1}^{j} \\
 \ldots & \ldots & \ldots & \ldots \\
 h_{n_j+m+1}^{j} & h_{n_j+m}^{j} & \ldots & h_{n_j-m+1}^{j}\\
 \end{array} \! \right),
 \end{equation*}
 
 \begin{equation*}
 F^j_-=
 \left( \! \begin{array}{cccc}
 h_{-n_j+m-1}^{j} & h_{-n_j+m-2}^{j} & \ldots & h_{-n_j-m-1}^{j} \\
 h_{-n_j+m-2}^{j} & h_{-n_j+m-3}^{j} & \ldots & h_{-n_j-m-2}^{j} \\
 \ldots & \ldots & \ldots & \ldots \\
 h_{-n_j+m-m_j}^{j} & h_{-n_j+m-m_j-1}^{j} & \ldots & h_{-n_j-m-m_j}^{j}\\
 \end{array} \! \right),
 \end{equation*}
 where $h_{p}^{j}=a_{p}^{j}/2$, $j=1,2,\ldots,k$; $p\in \mathbb{Z}$. 
 Then it is easy to verify that in equalities~\eqref{eq4.4}--\eqref{eq4.6} the coefficients at the powers of $e^{ipx}$ and
 $e^{-ipx}$ coincide: 
 $$
 d_{p}^j(n,\overrightarrow m)=d_{-p}^j(n,\overrightarrow m),\,\ p=1,2,\ldots;\,j=1,2,\ldots,k; \ u_l=u_{-l},\,\,\, l=1,2,\ldots,m.
 $$ 
 Therefore~\eqref{eq4.4}--\eqref{eq4.6} can be rewritten (taking into account the notation~\eqref{eq3.4} for $Q^t_m(x)$) in the form:
 $$
 Q^t_m(x;{\bf \tilde{f}^t})=u_0+
 \sum_{p=1}^{m} 2\,u_p\cos px\,,             
 $$
 $$
 P_j^t(x;{\bf \tilde{f}^t})=d^j_0(n,\overrightarrow m)+
 \sum_{p=1}^{n_j} 2\,d_{p}^j(n,\overrightarrow m)\cos px\,,                          
 $$
 $$	
 Q^t_m(x;{\bf \tilde{f}^t})f^t_j(x)-P^t_j(x;{\bf \tilde{f}^t})=\sum_{p=n+m+1}^{\infty} 2\, d_{p}^j(n,\overrightarrow m)\cos px\,.                                   
 $$
 If we replace $x$ with $\arccos x$ here, we get
 \begin{equation}
 \label{eq4.7}
 Q^{ch}_m(x;{\bf f^{ch}})=Q^t_m(\arccos x;{\bf \tilde{f}^t})=u_0+
 \sum_{p=1}^{m} 2\,u_p T_p(x)\,,            
 \end{equation}
 \begin{equation}
 \label{eq4.8}
 P_j^{ch}(x;{\bf f^{ch}})=P_j^t(\arccos x;{\bf \tilde{f}^t})=d^j_0+
 \sum_{p=1}^{n_j} 2\,d_{p}^j(n,\overrightarrow m)\,T_p(x)\,,                        
 \end{equation}
 \begin{equation}
 \label{eq4.9}	
 Q^{ch}_m(x;{\bf f^{ch}})f^{ch}_j(x)-P^{ch}_j(x;{\bf f^{ch}})=\sum_{p=n+m+1}^{\infty} 2\, d_{p}^j(n,\overrightarrow m)\,T_p(x)\,.                                   
 \end{equation}
 
 Note that if the multi-index $(n,\overrightarrow{m})\in \mathbb{Z}^{k+1}_+$ is not weakly normal for ${\bf \tilde{f}^t}$, then in this case the solution to problem ${\bf A^t}$ for the system of functions ${\bf \tilde{f}^t}$ is not unique, and therefore the solution to the corresponding problem ${\bf A^{ch}}$ is also not unique. Thus the following theorem is proved.
  
 \begin{theorem}\label{t6}
 	In order for problem $\bf A^{ch}$ to have a unique solution for a fixed multi-index $(n,\overrightarrow{m})$, $\overrightarrow{m}\neq (0,\ldots,0)$ and a system of functions $\bf f^{ch}$, defined by equalities~\eqref{eq4.1}, it is necessary and sufficient that the multi-index $(n,\overrightarrow{m})$ be weakly normal for the system $\bf \tilde{f}^t$.
 	
 	If the solution to problem $\bf A^{ch}$ is unique, then for a certain choice of the normalizing factor, representations~\eqref{eq4.7}--\eqref{eq4.9} are valid for its solutions.
 \end{theorem}
 
 \begin{remark}\label{r6}
 	Equalities~\eqref{eq4.7}--\eqref{eq4.9} define Hermite--Chebyshev polynomials and residual functions only for $x\in [-1, 1]$. For $x\notin [-1, 1]$ the values of the polynomials and residual functions are found using analytic continuation into the complex plane.
 \end{remark}
 
 Let us present several obvious corollaries of theorem~\ref{t6}.
 \begin{corollary}
 Let a multi-index $(n,\overrightarrow{m})$ be weakly normal for a system $\bf \tilde{f}^t$. Then the coefficients of the polynomials $Q^{ch}_m(x;{\bf f^{ch}}), P^{ch}_1(x;{\bf f^{ch}}),\ldots,P^{ch}_k(x;{\bf f^{ch}})$ are real numbers.
\end{corollary}

 \begin{corollary}
 	In order for problem $\bf A^{ch}$ to have a unique solution for any multi-index $(n,\overrightarrow{m})$,
 	it is necessary and sufficient that the system of functions $\bf \tilde{f}^t$ be weakly perfect.
\end{corollary}
 
 \begin{corollary}
 	\label{c6}
 	If the multi-index $(n,\overrightarrow{m})$ is weakly normal for the system of functions $\bf \tilde{f}^t$, then the linear Hermite--Chebyshev approximations
 	$$
 	\pi^{ch}_j(x;{\bf f^{ch}})=\frac{P^{ch}_j(x;{\bf f^{ch}})}{Q^{ch}_m(x;{\bf f^{ch}})},\,\,\,
 	j=1,2,\ldots,k\,
 	$$	
 	by relations~\eqref{eq4.2} are determined uniquely.
\end{corollary}

 In addition to corollary \ref{c6}, we note that in work~\cite{Ibr} for $k=1$, sufficient conditions for the uniqueness of linear Hermite--Chebyshev approximations (for $k=1$ in~\cite{Ibr} they are called {\it linear Pad\'e--Chebyshev approximations}) were obtained as a result of describing the structure of the kernel of some {\it Toeplitz--plus--Hankel} matrices, the elements of which are the coefficients of the series $f^{ch}_1(x)$. In particular, in~\cite{Ibr} it is established that for the uniqueness of linear Pad\'e--Chebyshev approximations it is sufficient that the corresponding Toeplitz--plus--Hankel matrix has full rank.

 \subsection{Nonlinear Hermite--Chebyshev approximations}\label{s4.2} 
   Let us now define another construction of rational functions of consistent approximation of series~\eqref{eq4.1}.
  
  \begin{definition}\label{de12}
  	Rational functions of the form
  	$$
  	\widehat{\pi}^{ch}_j(x;{\bf f^{ch}})=	\widehat{\pi}^{ch}_{n_j,n,\overrightarrow{m}}(x;{\bf f^{ch}})=\frac{\widehat{P}^{ch}_{j}(x)}{\widehat{Q}^{ch}_m(x)},
  	$$
  	where the polynomials
  	$$\widehat{Q}^{ch}_m(x)=\widehat{Q}^{ch}_{n,\overrightarrow{m}}(x;{\bf f^{ch}})=\sum_{p=0}^{m}\widehat{u}_pT_p(x),
  	$$
  	$$\widehat{P}^{ch}_{j}(x)=\widehat{P}^{ch}_{n_j,n,\overrightarrow{m}}(x;{\bf f^{ch}})=\sum_{p=0}^{n_j}\widehat{v}^j_p\,T_p(x)\,\,\, (n_j=n+m-m_j)
  	$$ 
  	are chosen so that for $j=1,2,\ldots,k$
  	$$
  	f^{ch}_j(x)-\frac{\widehat{P}^{ch}_{j}(x)}{\widehat{Q}^{ch}_m(x)}=
  	\sum_{l=n+m+1}^{\infty} \widehat{a}^j_l\,T_l(x),
  	$$
  	will be called {\it nonlinear Hermite--Chebyshev approximations} for the multi-index $(n,\overrightarrow{m})$ and the system ${\bf f^{ch}}$.	
  \end{definition}

  In the case $k=1$, there exist functions that can be expanded in a series of Chebyshev polynomials for which nonlinear Hermite--Chebyshev approximations do not exist, and there exist functions for which nonlinear Hermite--Chebyshev approximations exist but do not coincide with linear Hermite--Chebyshev approximations (\cite{ Suet2,  Bek, SuetCh, Suet1,  SuetDis}).  
  
   In \cite{StarKechOsn} for $k\geqslant 1$, examples of systems of functions are constructed for which there are nonlinear Hermite--Chebyshev approximations that do not coincide with linear ones.
   Based on the results of the previous paragraph, we introduce into consideration new systems of functions that have similar properties. 
  
   Let us consider the system of functions ${\bf F_{k,\gamma}}=\{F_{\gamma}(x;\lambda_j)\}^k_{j=1}$, where
  $$
  F_{\gamma}(x;\lambda)=
  \sum_{l=0}^{\infty}\dfrac{\lambda^l}{(\gamma)_l}\,T_l(x)\,,
  $$
  and $\lambda\in \mathbb{R}$, $\gamma\in\mathbb{R}\setminus
  \mathbb{Z_{-}}$, $\{\lambda_j\}^k_{j=1}$ is a set of distinct non-zero real numbers. 
  
  The following theorem is proved similarly to theorem~\ref{t8} using the same reasoning that was used in the proof of theorem~\ref{t6}.
  
  \begin{theorem}\label{t10}
  	For a multi-index
  	$(n,\overrightarrow{m})\in \mathbb{Z}^{k+1}_+$, $\overrightarrow{m}\neq (0,\ldots,0)$, satisfying the condition $n\geqslant m_j$ $(j=1,2,\ldots,k)$, for all $n$, starting from some $n_0$, there exist nonlinear Hermite--Chebyshev approximations $\{{\widehat{\pi}}^{ch}_j(x;{\bf F_{k,\gamma}})\}_{j=1}^k$, and with the appropriate normalization the following representations of their denominator are valid:
  	$$
  	\widehat{Q}^{ch}_m(x;{\bf F_{k,\gamma}})=
  	{Q}_m(e^{i\arccos  x};{\bf E_{k,\gamma}})\,\cdotp\, \overline{{Q}_m(e^{i\arccos  x};{\bf E_{k,\gamma}}})=
  	$$
  	$$
  	=\frac{1}{\Gamma^2(n+m+\gamma)}\left |\int\limits_0^{+\infty} T_{\gamma}(t) e^{-t( x+i\sqrt{1-x^2})} dt\right |^2. 
  	$$  	
    If, in addition, $0<\lambda_1<\lambda_2<\ldots<\lambda_k$, then among $\{{\widehat{\pi}}^{ch}_j(x;{\bf F_{k,\gamma}})\}_{j=1}^k$ there is a nonlinear Hermite--Chebyshev approximation ${\widehat{\pi}}^{ch}_j(x;{\bf F_{k,\gamma}})$ that is not a linear Hermite--Chebyshev approximation $\pi^{ch}_j(x;{\bf F_{k,\gamma}})$.
  \end{theorem}

\begin{remark}\label{r6}
In work~\cite{Aptek81M-L} integral representations were obtained only for the denominators ${Q}_m(z;{\bf E_{k,\gamma}})$ and residual functions of the Hermite--Pad\'e approximations of the Mittag-Leffler functions. Therefore, in theorems~\ref{t8} and \ref{t10}, formulas for the numerators of the nonlinear Hermite--Fourier and Hermite--Chebyshev approximations are missing. At the same time, the numerators $\widehat{P}^t_{j}(x;{\bf G_{k,\gamma}})$, $\widehat{P}^{ch}_{j}(x;{\bf F_{k,\gamma}})$ can also be written out in explicit form if we use equality~\eqref{eq3.31} and the representations of the polynomials $P^j_{n_j}(z;{\bf E_{k,\gamma}})$ obtained in~\cite{StarRjab2} in the form of determinants, the elements of which are the Taylor coefficients of the functions of the system ${\bf E_{k,\gamma}}$: if the conditions of theorems 4 and 6 are met, then for $j=1,2,\ldots,k$
$$
\widehat{P}^t_{j}(x;{\bf G_{k,\gamma}})=Re\left\{P^j_{n_j}(e^{ix};{\bf E_{k,\gamma}})\,\cdotp\,\overline{{Q}_m(e^{ix};{\bf E_{k,\gamma}})}\right\},
$$
$$
\widehat{P}^{ch}_{j}(x;{\bf F_{k,\gamma}})= Re\left\{P^j_{n_j}(e^{i\arccos  x};{\bf E_{k,\gamma}})\,\cdotp\,\overline{{Q}_m(e^{i\arccos  x};{\bf E_{k,\gamma}})}\right\}.
$$
\end{remark}

\vskip0.3truecm

\fullauthor{Aleksandr Pavlovich Starovoitov}
\address {F. Skorina Gomel State University,\\
104 Sovetskaya str.,
Gomel, 246019 Republic of Belarus,}
\email{apsvoitov@gmail.com, svoitov@gsu.by}

\fullauthor{Igor Viktorovich  Kruglikov}
\address {F. Skorina Gomel State University,\\
	104 Sovetskaya str.,
	Gomel, 246019 Republic of Belarus,}
\email{igor.v.kruglikov@gmail.com}

\fullauthor{Tatyana Mikchyalovna Osnach}
\address {F. Skorina Gomel State University,\\
	104 Sovetskaya str.,
	Gomel, 246019 Republic of Belarus,}
\email{osnach@gsu.by}
\label{lastpage}

\end{document}